\documentclass[10pt,a4paper]{amsart}
\usepackage{amssymb,amsmath,amsopn,amsthm,graphicx}

\begin{document}

\newtheorem{definition}{Definition}[section]
\newtheorem{definitions}[definition]{Definitions}
\newtheorem{deflem}[definition]{Definition and Lemma}
\newtheorem{lemma}[definition]{Lemma}
\newtheorem{pro}[definition]{Proposition}
\newtheorem{theorem}[definition]{Theorem}
\newtheorem{cor}[definition]{Corollary}
\newtheorem{cors}[definition]{Corollaries}
\theoremstyle{remark}
\newtheorem{remark}[definition]{Remark}
\theoremstyle{remark}
\newtheorem{remarks}[definition]{Remarks}
\theoremstyle{remark}
\newtheorem{notation}[definition]{Notation}
\theoremstyle{remark}
\newtheorem{example}[definition]{Example}
\theoremstyle{remark}
\newtheorem{examples}[definition]{Examples}
\theoremstyle{remark}
\newtheorem{dgram}[definition]{Diagram}
\theoremstyle{remark}
\newtheorem{fact}[definition]{Fact}
\theoremstyle{remark}
\newtheorem{illust}[definition]{Illustration}
\theoremstyle{remark}
\newtheorem{rmk}[definition]{Remark}
\theoremstyle{definition}
\newtheorem{que}[definition]{Question}
\theoremstyle{definition}
\newtheorem{conj}[definition]{Conjecture}

\renewenvironment{proof}{\noindent {\bf{Proof.}}}{\hspace*{3mm}{$\Box$}{\vspace{9pt}}}
\title{\textsc{On the Grothendieck ring of varieties}}
\maketitle
\begin{center}
		AMIT KUBER\footnote{Email \texttt{amit.kuber@postgrad.manchester.ac.uk}.
		Research partially supported by a School of Mathematics, University of Manchester Scholarship.},

		\medskip

		School of Mathematics, \ University of Manchester, \\
		Manchester M13 9PL, \ England.
	\end{center}

\begin{abstract}
Let $\operatorname{K}_0(\operatorname{Var}_k)$ denote the Grothendieck ring of $k$-varieties over an algebraically closed field $k$. Larsen and Lunts asked if two $k$-varieties having the same class in $\operatorname{K}_0 (\operatorname{Var}_k)$ are piecewise isomorphic. Gromov asked if a birational self-map of a $k$-variety can be extended to a piecewise automorphism. We show that these two questions are equivalent over any algebraically closed field. If these two questions admit a positive answer, then we prove that its underlying abelian group is a free abelian group. Furthermore, if $\mathfrak B$ denotes the multiplicative monoid of birational equivalence classes of irreducible $k$-varieties then we also prove that the associated graded ring of the Grothendieck ring is the monoid ring $\mathbb Z[\mathfrak B]$.
\end{abstract}

\section{Introduction}\label{intro}
\newcommand\Q{\operatorname{K}_0(\mathcal S_k)}
\newcommand\V[1][]{\operatorname{Var}^{#1}_{k}}
\newcommand\K{\operatorname{K}_0(\V)}
\newcommand\G{\operatorname{K}^+_0(\V)}
\newcommand\D{\overline{\operatorname{Constr}}(k)}
\newcommand\F[1][]{\operatorname{Constr}^{#1}(k)}
Let $k$ be a field. A variety over $k$ is a reduced separated scheme of finite type. A subvariety of a variety $X$ is said to be locally closed if it can be written as the intersection of an open subvariety with a closed subvariety. Let $\V$ denote the category of $k$-varieties. The Grothendieck group $\G$ is the quotient of the free abelian group generated by the isomorphism classes of $k$-varieties by the following relations.
\begin{equation}\label{genrelpre}
[X]-[Y]=[X\setminus Y]\mbox{ whenever }Y\subseteq X\mbox{ is a closed subvariety}.
\end{equation}
It can be given a ring structure by taking the reduced product $(X\times_{\mathrm{Spec}\,k}Y)_{red}$ of varieties. Recall that if the field $k$ is algebraically closed, then we can simply talk about the product $X\times_{\mathrm{Spec}\,k}Y$. We denote the Grothendieck ring of varieties by $\K$.

Recall that if $R$ is a commutative ring, then an $R$-valued motivic measure is a ring homomorphism $\K\to R$. The Grothendieck ring plays an important role in the theory of motivic integration being the value ring of the universal motivic measure on $k$-varieties. But very little is known about this ring. Poonen \cite[Theorem~1]{Poonen} and Koll\'{a}r \cite[Example~6]{Kol} show that this ring is not a domain when $k$ has characteristic $0$.

Characterizing equality in the Grothendieck ring is an important issue. We need some notation to state this problem precisely.

Set $[n]=\{1,2,\hdots,n\}$ for each $n\in\mathbb Z_{>0}$. Two varieties $X$ and $Y$ are said to be \textbf{piecewise isomorphic}, written $X\doteqdot Y$, if there are partitions $X=\bigsqcup_{i\in[n]}X_i$ and $Y=\bigsqcup_{j\in[n]}Y_j$ of $X$ and $Y$ into locally closed subvarieties such that there is a permutation $\sigma$ of $[n]$ with $X_i$ isomorphic to $Y_{\sigma(i)}$ as a variety. 

If $X\doteqdot Y$, then clearly $[X]=[Y]$ in $\K$. Larsen and Lunts asked whether the converse is true.
\begin{que}\protect{\cite[Question~1.2]{LL}}\label{LLQ}
Suppose $X$ and $Y$ are two $k$-varieties such that $[X]=[Y]$ in $\K$. Is it true that $X\doteqdot Y$?
\end{que}
In the case when $k$ is algebraically closed, we reformulate this question as the cancellative property of the Grothendieck semiring $\mathcal S_k$ of piecewise isomorphic classes of $k$-varieties in Question \ref{semiringcancel}.

Liu and Sebag answered this question over an algebraically closed field of characteristic $0$ for varieties with dimension at most one (\cite[Propositions~5,\,6]{LS}) and for some classes of dimension two varieties (\cite[Theorems~4,\,5]{LS}). Sebag extended this result further (\cite[Theorem~3.3]{Seb}).

This question is quite natural and has many important applications to birational geometry. Consider the following question asked by Gromov as an example.
\begin{que}\protect{\cite[\S 3.G$'''$]{Gro}}\label{G}
Let $X$ and $Y$ be algebraic varieties which admit an embedding into a third one, say $X\hookrightarrow Z$ and $Y\hookrightarrow Z$, such that the complements $Z\setminus X$ and $Z\setminus Y$ are biregularly isomorphic. How far are $X$ and $Y$ from being birationally equivalent? Under what conditions are $X$ and $Y$ piecewise isomorphic?
\end{que}

Lamy and Sebag \cite{LamSeb} studied the following conjectural reformulation of this question in characteristic $0$.
\begin{conj}[see \protect{\cite[Conjecture~1]{LamSeb}}]\label{GG}
Let $k$ be an algebraically closed field and let $X$ be a $k$-variety. Let $\phi:X\dashrightarrow X$ be a birational map. Then it is possible to extend the map $\phi$ to a piecewise automorphism of $X$.
\end{conj}
It is known (see \cite{LamSeb}) that a positive answer to Question \ref{LLQ} will settle this conjecture in the affirmative. We prove the converse in Theorem \ref{GRRMGM} showing that the two statements are in fact equivalent. This may be known to the experts, but we could not find a proof in the literature.

Larsen and Lunts \cite{LL} obtained an important motivic measure described in the following theorem.
\begin{theorem}[\protect{\cite[Theorem~2.3]{LL}}]
Suppose $k$ is an algebraically closed field of characteristic $0$. Let $\mathfrak{sb}$ denote the multiplicative monoid of stable birational equivalence classes of irreducible varieties. There exists a unique surjective ring homomorphism $\Psi:\K\to\mathbb Z[\mathfrak{sb}]$ that assigns to the class in $\K$ of a smooth irreducible proper variety its stable birational equivalence class in $\mathbb Z[\mathfrak{sb}]$.
\end{theorem}

Bittner \cite{Bit} obtained the following presentations of the Grothendieck group. Larsen and Lunts mention that Bittner's presentation subsumes the theorem above \cite[Remark~2.4]{LL} and this assertion has been proved in detail by Sahasrabudhe in \cite{Sah}.
\begin{theorem}[\protect{\cite[Theorem~3.1]{Bit}}]\label{Bittner}
Let $k$ be a field of characteristic $0$. The Grothendieck group $\G$ has the following presentations:
\begin{itemize}
\item[(sm)] as the abelian group generated by the isomorphism classes of smooth varieties over $k$ subject to the relations $[X]=[Y]+[X\setminus Y]$, where $X$ is smooth and $Y\subseteq X$ is a smooth closed subvariety;
\item[(bl)] as the abelian group generated by the isomorphism classes of smooth complete $k$-varieties subject to the relations $[\emptyset]=0$ and $[\mathrm{Bl}_Y\,X]-[E]=[X]-[Y]$, where $X$ is smooth and complete, $Y\subseteq X$ is a smooth closed subvariety, $\mathrm{Bl}_Y\,X$ is the blow-up of $X$ along $Y$ and $E$ is the exceptional divisor of this blow-up.
\end{itemize}
\end{theorem}

In Theorem \ref{GGISFREE}, we show that if Question \ref{LLQ} admits a positive answer over an algebraically closed field $k$ then the Grothendieck group $\G$ is a free abelian group. Further if $k$ has characteristic $0$, then this result subsumes Bittner's presentation in view of Hironaka's theorem on resolution of singularities.

\textbf{Conventions:} In the sequel $k$ denotes an algebraically closed field unless otherwise mentioned. If $X$ is a variety, we use $\dim X$ to denote its dimension and $d(X)$ to denote the number of its irreducible components of maximal dimension.

\section{The semiring of piecewise isomorphism classes of varieties}\label{Grsemiring}
The \textbf{Grothendieck group} $\mathrm{K}_0(M)$ associated with a commutative monoid $(M,+,0)$ is an abelian group equipped with a monoid homomorphism $q:M\to\mathrm{K}_0(M)$ that satisfies the following universal property: given an abelian group $N$ and a monoid homomorphism $f:M\to N$, there is a unique group homomorphism $\overline{f}:\mathrm{K}_0(M)\to N$ such that $f=\overline{f}\circ q$.

A commutative monoid $M$ is said to be \textbf{cancellative} if $a+c=b+c\Rightarrow a=b$ for all $a,b,c\in M$. It is easy to see that $M$ is cancellative if and only if the map $q:M\to\mathrm{K}_0(M)$ is injective.

If $(M,+,\cdotp,0,1)$ is a commutative semiring, then the abelian group $\mathrm{K}_0(M)$ can be equipped with a natural commutative ring structure. In this case, we say that $\mathrm{K}_0(M)$ is the \textbf{Grothendieck ring} of the semiring $M$.

Let $\{A\}$ denote the piecewise isomorphism class of a variety $A$. The set $\mathcal S_k$ of piecewise isomorphism classes of $k$-varieties carries a natural semiring structure.
\begin{eqnarray*}
0&:=&\{\emptyset\},\\
\{A\}+\{B\}&:=&\{A'\sqcup B'\}\mbox{ where }A'\in\{A\},\,B'\in\{B\},\,A'\cap B'=\emptyset.
\end{eqnarray*}
The product of the classes of varieties is defined by $\{A\}\cdotp\{B\}:=\{(A\times_{\mathrm{Spec}\,k}B)_{red}\}$, where $\{\mathrm{Spec}\,k\}$ is the multiplicative identity.

A general element of $\Q$ can be written as $\{A\}-\{B\}$ for some varieties $A,B$. Furthermore, $\{A_1\}-\{B_1\}=\{A_2\}-\{B_2\}$ if and only if there is some variety $C$ such that $A'_1\sqcup B'_2\sqcup C\doteqdot A'_2\sqcup B'_1\sqcup C$ in $\mathcal S_k$ for some $A'_j\in\{A_j\},B'_j\in\{B_j\}$ for $j=1,2$ such that $A'_1,B'_2,C$ and $A'_2,B'_1,C$ are families of pairwise disjoint varieties.

On the other hand, a general element of $\K$ can be expressed as a finite linear combination $\sum_ia_i[A_i]-\sum_jb_j[B_j]$ with $a_i,b_j\in\mathbb Z^+$ and $A_i,B_j\in\V$. We can choose some $A'_{i1},A'_{i2},\hdots,A'_{ia_i}\in[A_i]$ and $B'_{j1},B'_{j2},\hdots,B'_{jb_j}\in[B_j]$ for each $i,j$ such that every two distinct $A'_{ik}$ and $B'_{jl}$ are disjoint. Let $A:=\bigsqcup_{i,k}A'_{ik}$ and $B:=\bigsqcup_{j,l}B'_{jl}$. Then the identities $[A]=\sum_ia_i[A_i]$ and $[B]=\sum_jb_j[B_j]$ are clearly true in $\K$. Therefore a general element of $\K$ can be expressed as $[A]-[B]$ for some varieties $A,B$.
\begin{pro}\label{isoGR}
Let $k$ be an algebraically closed field. Then the natural map $\psi:\K\rightarrow\Q$ defined by $\psi([A]-[B]):=\{A\}-\{B\}$ is an isomorphism of rings.
\end{pro}

\begin{proof}
Recall that piecewise isomorphic varieties have the same class in $\K$. Note as a consequence of \eqref{genrelpre} that if $Z_1$ and $Z_2$ are two disjoint varieties, then $[Z_1]=[Z_2]$ if and only if there is some $W$ disjoint from both $Z_1$ and $Z_2$ such that $Z_1\sqcup W\doteqdot Z_2\sqcup W$.

Let $A_1,B_1,A_2$ and $B_2$ be pairwise disjoint varieties.

Now $[A_1]-[B_1]=[A_2]-[B_2]$ in $\K$

$\iff[A_1]+[B_2]=[A_2]+[B_1]$ in $\K$

$\iff$ there is a variety $C$ disjoint from all $A_i$ and $B_j$

$\ \ \ \ \ \ \ \ $such that $A_1\sqcup B_2\sqcup C\doteqdot A_2\sqcup B_1\sqcup C$

$\iff\{A_1\}+\{B_2\}=\{A_2\}+\{B_1\}$ in $\mathcal S_k$

$\iff\{A_1\}-\{B_1\}=\{A_2\}-\{B_2\}$ in $\Q$.

Thus the map $\psi$ is both well-defined and injective. It is clearly surjective and preserves addition.

Finally we note that $\psi$ also preserves multiplication. Observe that for any two varieties $X$ and $Y$, we have $[X]\cdotp[Y]=[X\times_{\mathrm{Spec\,k}}Y]$ in $\K$ and $\{X\}\cdotp\{Y\}=\{X\times_{\mathrm{Spec\,k}}Y\}$ in $\Q$. Hence $\psi$ preserves multiplication of varieties. Using the distributivity of multiplication over addition completes the proof.
\end{proof}

The following question is natural.
\begin{que}\label{semiringcancel}
Let $k$ be an algebraically closed field. Is the semiring $\mathcal S_k$ of piecewise isomorphism classes of $k$-varieties cancellative?
\end{que}
A positive answer to this question is equivalent to injectivity of the natural map $q:\mathcal S_k\rightarrow \Q$. In view of Proposition \ref{isoGR}, it is also equivalent to injectivity of the map $\psi^{-1}\circ q:\mathcal S_k\rightarrow\K$. Hence a positive answer to Question \ref{semiringcancel} is equivalent to a positive answer to Question \ref{LLQ}.

\section{Equivalence of Question \ref{LLQ} and Conjecture \ref{GG}}\label{maintheorem}
We note a consequence of equality in the Grothendieck ring from \cite{LS}. Scanlon has pointed out another proof using counting function methods from \cite{Kra}.
\begin{pro}[\protect{\cite[Cor.\,5]{LS}}]\label{dimdeg}
Let $k$ be an algebraically closed field of characteristic $0$. Let $A$ and $B$ be two varieties with $[A]=[B]$ in $\K$. Then $\dim A=\dim B$ and $d(A)=d(B)$.
\end{pro}

Given two varieties $V$ and $W$ such that $[V]=[W]$ in $\K$, Proposition \ref{isoGR} states that there is some variety $Z$ disjoint from $V$ and $W$ such that $V\sqcup Z\doteqdot W\sqcup Z$. Under the hypothesis of Conjecture \ref{GG}, we develop a technique in Proposition \ref{matching} to remove a dense subset of $Z$ from both $V\sqcup Z$ and $W\sqcup Z$ to leave piecewise isomorphic complements. In fact the following Proposition is a reformulation of Conjecture \ref{GG}.

\begin{pro}\label{matching}
Suppose Conjecture \ref{GG} holds for an algebraically closed field $k$.

Let $V,W$ and $Z$ be $k$-varieties such that $Z$ is disjoint from both $V$ and $W$, $\dim V\leq\dim W\leq t=\dim Z$ and $d(Z)=e$. Assume that $d(V)=d(W)$ if $\dim V=\dim W$. Further let $d=\begin{cases}d(V)&\mbox{ if }\dim V=\dim W=t,\\0&\mbox{ otherwise}.\end{cases}$

Let $S_1,S_2,\hdots,S_{d+e}$ and $T_1,T_2,\hdots,T_{d+e}$ be families of pairwise disjoint irreducible subvarieties of $V\sqcup Z$ and $W\sqcup Z$ respectively such that $\dim S_l=\dim T_l=t$ for each $l\in[d+e]$. Assume that the varieties $S_l$ and $T_l$ are either disjoint from or contained in $Z$ for each $l\in[d+e]$. Furthermore assume that $\tau$ is a permutation of $[d+e]$ such that $f_l:S_l\cong T_{\tau(l)}$ is a variety isomorphism for each $l\in[d+e]$.

Then there are subsets $P,Q\subseteq[d+e]$ of size $e$, a bijection $\lambda:Q\to P$ and dense subvarieties $S'_l\subseteq S_l,\ T'_l\subseteq T_l$ for $l\in[d+e]$ such that the following hold:
\begin{itemize}
\item $S'_{\lambda(l)}=T'_l\subseteq Z$ for each $l\in Q$;
\item $\bigsqcup_{m\notin P}S'_m\doteqdot\bigsqcup_{l\notin Q}T'_l$;
\item $\bigsqcup_{l\in[d+e]}(S_l\setminus S'_l)\doteqdot\bigsqcup_{l\in[d+e]}(T_l\setminus T'_l)$.
\end{itemize}
\end{pro}

\begin{proof}
We have $S_l\subseteq Z$ and $T_m\subseteq Z$ for exactly $e$ values of both $l$ and $m$. Let $Q,P\subseteq[d+e]$ be the sets of such $m$ and $l$ respectively. For each $l\in P$, there is a unique $m\in Q$ such that $\dim(S_l\cap T_m)=t$. Let $\lambda:Q\to P$ define this correspondence.

\textbf{Case I:} Suppose that $\lambda(l)=\tau^{-1}(l)$ for each $l\in Q$.

In this case we set $S'_{\lambda(l)}=T'_l:=S_{\lambda(l)}\cap T_l$ for each $l\in Q$. The isomorphism $f_{\lambda(l)}:S_{\lambda(l)}\to T_l$ can be seen as a birational self-map of $S_{\lambda(l)}\cup T_l$. Since Conjecture \ref{GG} holds, this birational map can be extended to obtain a piecewise automorphism of $S_{\lambda(l)}\cup T_l$. In particular, one gets a piecewise isomorphism $T_l\setminus S_{\lambda(l)}\doteqdot S_{\lambda(l)}\setminus T_l$ of lower dimensional subvarieties.

\textbf{Case II:} Suppose that $\lambda(i)\neq\tau^{-1}(i)$ for some $i\in Q$. Fix such $i$ and let $j:=\lambda(i)$.

The idea of the proof is to find subvarieties $S^1_l\subseteq S_l$ and $T^1_l\subseteq T_l$ for each $l\in[d+e]$ and a permutation $\tau_1$ of $[d+e]$ such that the following properties are satisfied:
\begin{itemize}
\item[(i)] $\{l\in Q:\lambda(l)=\tau_1^{-1}(l)\}\supsetneq\{l\in Q:\lambda(l)=\tau^{-1}(l)\}$;
\item[(ii)] $f_l:S_l\setminus S^1_l\cong T_{\tau(l)}\setminus T^1_{\tau(l)}$ is an isomorphism for each $l\in[d+e]$;
\end{itemize}
and then continue inductively.

Since $j\neq\tau^{-1}(i)$, we make the following assignments.

$S^1_l:=\begin{cases}T_i\cap S_j&\mbox{if }l=j,\\f_{\tau^{-1}(i)}^{-1}(T_i\cap S_j)&\mbox{if }l=\tau^{-1}(i)\\S_l&\mbox{otherwise}.\end{cases}$

$T^1_l:=\begin{cases}f_j(T_i\cap S_j)&\mbox{if }l=\tau(j),\\T_i\cap S_j&\mbox{if }l=i\\T_l&\mbox{otherwise}.\end{cases}$

The maps $f_{\tau^{-1}(i)}$ and $f_j$ clearly restrict to isomorphisms $S_{\tau^{-1}(i)}\setminus S^1_{\tau^{-1}(i)}\cong T_i\setminus T^1_i$ and $S_j\setminus S^1_j\cong T_{\tau(j)}\setminus T^1_{\tau(j)}$ of lower dimensional subvarieties. This takes care of property $(ii)$.

Now we define $\tau_1:[d+e]\to[d+e]$ as follows.

$\tau_1(l):=\begin{cases}i&\mbox{if }l=j,\\\tau(j)&\mbox{if }l=\tau^{-1}(i)\\\tau(l)&\mbox{otherwise}.\end{cases}$

Note that $\lambda(i)\neq\tau^{-1}(i)$, but $\lambda(i)=\tau_1^{-1}(i)$. This shows that $(i)$ holds.

Furthermore, $f^1_{\tau^{-1}(i)}:=f_j\circ f_{\tau^{-1}(i)}:S^1_{\tau^{-1}(i)}\to T^1_{\tau_1(\tau^{-1}(i))}$ and $f^1_j:=id:S^1_j\to T^1_{\tau_1(j)}$ are isomorphisms. For the remaining $l\in[d+e]$, we set $f^1_l:=f_l$.

Thus $f^1_l:S^1_l\to T^1_{\tau_1(l)}$ is an isomorphism for each $l\in[d+e]$. If $\lambda$ does not agree with $\tau_1^{-1}$ on $Q$, we iterate the process with varieties $S^1_l,\,T^1_l$, functions $f^1_l$ and permutation $\tau_1$ until some $(\tau_n)^{-1}$ agrees with $\lambda$ on $Q$.

We set $T'_l:=T^n_l$ for each $l\notin Q$ and $S'_l:=S^n_l$ for each $l\notin P$.

The varieties $S^n_{\lambda(l)},\,T^n_l$, for $l\in Q$, together with the function $\lambda=\tau_n^{-1}\upharpoonright_Q$ is the set-up for the first case. A construction similar to that case gives the required varieties $S'_{\lambda(l)},\,T'_l$.

In both cases, it is clear that the construction guarantees the final two conditions in the statement of the proposition.
\end{proof}

Now we derive the result of Proposition \ref{dimdeg} for an algebraically closed field of non-zero characteristic under the hypothesis of Conjecture \ref{GG}.

\begin{cor}\label{dimdeggen}
Suppose Conjecture \ref{GG} holds for an algebraically closed field $k$. Let $V$ and $W$ be two $k$-varieties with $[V]=[W]$ in $\K$. Then $\dim V=\dim W$ and $d(V)=d(W)$.
\end{cor}

\begin{proof} 
Without loss we may assume that $V$ and $W$ are disjoint and have the same class in $\K$. In view of Proposition \ref{isoGR}, there is a variety $Z$ disjoint from both $V$ and $W$ such that $V\sqcup Z\doteqdot W\sqcup Z$.

\textbf{Case I:} $\dim Z,\dim V\leq\dim W$. 

Counting the number of irreducible components of maximal dimension on both sides of the piecewise isomorphism $V\sqcup Z\doteqdot W\sqcup Z$ forces the equalities $\dim V=\dim W$ and $d(V)=d(W)$.

\textbf{Case II:} $\dim V\leq\dim W<\dim Z$. 

Let $e:=d(Z)$ and $d=0$. The piecewise isomorphism $V\sqcup Z\doteqdot W\sqcup Z$ gives two families $S_1,S_2,\hdots,S_e$ and $T_1,T_2,\hdots,T_e$ of irreducible subvarieties of $Z$, each of dimension $e$, satisfying the hypotheses of Proposition \ref{matching}. 

An application of the Proposition then gives a subvariety $Z_1=\bigsqcup_{i\in[e]}T'_i\subseteq Z$ with $\dim Z_1=\dim Z, d(Z_1)=d(Z)$ such that $V\sqcup (Z\setminus Z_1)\doteqdot W\sqcup (Z\setminus Z_1)$. Proposition \ref{matching} can be used repeatedly until the equality $[V]=[W]$ is witnessed by a subvariety $Z'\subseteq Z$ satisfying $\dim Z'\leq\dim W$. Then we land up in the first case completing the proof. 
\end{proof}

\begin{theorem}\label{GRRMGM}
Let $k$ be an algebraically closed field. If Conjecture \ref{GG} holds for $k$, then Question \ref{LLQ} admits a positive answer over $k$.
\end{theorem}

\begin{proof}
Let $V$ and $W$ be two varieties with $[V]=[W]$ in $\K$. Then Proposition \ref{isoGR} states that there is a variety $Z$ of dimension $t$ and $d(Z)=e$, say, disjoint from both $V$ and $W$, which witnesses this equality, i.e., $V\sqcup Z\doteqdot W\sqcup Z$. Proposition \ref{dimdeggen} then gives $\dim V=\dim W=:s$ and $d(V)=d(W)$.

If $s=0$, then $d(V)=d(W)$ implies $V\doteqdot W$.

If $s>0$, then we describe a procedure to reduce the sum $s+t$ in two different cases.

\textbf{Case I:} Suppose $s\leq t$. Let $d=\begin{cases}d(V)&\mbox{if }s=t,\\0&\mbox{otherwise.}\end{cases}$

The piecewise isomorphism $V\sqcup Z\doteqdot W\sqcup Z$ gives two families $S_1,S_2,\hdots,S_{d+e}$ and $T_1,T_2,\hdots,T_{d+e}$ of irreducible subvarieties of $V\sqcup Z$ and $W\sqcup Z$ respectively of maximal dimension satisfying the hypotheses of Proposition \ref{matching}. Furthermore, we also get that $(V\sqcup Z)\setminus\left(\bigsqcup_{i\in[d+e]}S_i\right)\doteqdot(W\sqcup Z)\setminus\left(\bigsqcup_{i\in[d+e]}T_i\right)$.

If Conjecture \ref{GG} holds, we can apply Proposition \ref{matching} to obtain a dense subvariety $Z_1:=\bigsqcup_{i\in Q}T'_i\subseteq Z$. The other conclusions of the proposition give the following properties:
\begin{itemize}
\item the variety $Z':=Z\setminus Z_1$ witnesses $[V]=[W]$, i.e., $V\sqcup Z'\doteqdot W\sqcup Z'$;
\item $\dim Z'<\dim Z$.
\end{itemize}

The use of Proposition \ref{matching} can be repeated if $\dim Z'\geq s$. Hence the equality $[V]=[W]$ in $\K$ is witnessed by some variety $Z''$ of dimension less than $s$.

\textbf{Case II:} Suppose $s>t$. In this case, the piecewise isomorphism $V\sqcup Z\doteqdot W\sqcup Z$ gives $V'\subset V$, $W'\subset W$ with $\dim V'=\dim W'<s$ such that $V\setminus V'\doteqdot W\setminus W'$ and $V'\sqcup Z\doteqdot W'\sqcup Z$.

The two cases complete the proof that $V\doteqdot W$.
\end{proof}

\section{A presentation for $\G$ under Conjecture \ref{GG}}\label{ggisfree}
We assume that the Conjecture \ref{GG} holds (equivalently, in view of Theorem \ref{GRRMGM}, Question \ref{LLQ} admits a positive answer) for an algebraically closed field $k$ in this section.

For each $n\in\mathbb Z_{\geq0}$ let $\V[n]$ denote the proper class of $k$-varieties of dimension at most $n$. Then $\{\V[n]\}_{n\geq0}$ is a filtration on the objects of $\V$. Further let $S_n$ denote the monoid, under $\sqcup$, of piecewise isomorphism classes of varieties in $\V[n]$ and $H_n$ denote the Grothendieck group associated with $S_n$ for each $n\geq0$. If Conjecture \ref{GG} holds, then $H_n$ is the subgroup of $\K$ generated by $S_n$ and thus, for each $n\in\mathbb Z_{\geq0}$, the natural map $H_n\to H_{n+1}$ is injective.

Let $\mathfrak M$ denote a set of representatives of birational equivalence classes of irreducible varieties. Then $\mathfrak M=\bigsqcup_{n\in\mathbb Z_{\geq0}} \mathfrak M_n$, where $\mathfrak M_n$ is the set of all dimension $n$ varieties in $\mathfrak M$. We use $\mathcal A,\mathcal B$ etc. to denote the elements of $\mathfrak M$.

We say that a variety $A$ of dimension $n$ is \textbf{$\mathfrak M$-admissible} (or just \textbf{admissible}) if it can be embedded into some $\mathcal A\in\mathfrak M_n$. The assignment $A\mapsto\mathcal A$ is a well-defined and dimension preserving map on the class of admissible varieties. Note that every admissible variety has a unique irreducible component of maximal dimension. We say that a \textbf{partition} $D=\bigsqcup_{i\in[m]}D_i$ of a variety $D$ into locally closed subvarieties \textbf{is admissible} if each $D_i$ is admissible. Note that each variety admits an admissible partition.

\begin{theorem}\label{GGISFREE}
Suppose Question \ref{LLQ} admits a positive answer over an algebraically closed field $k$. Let $\mathfrak M$ denote a set of representatives of birational equivalence classes of irreducible $k$-varieties. Then there is a unique group isomorphism $ev_{\mathfrak M}:\G\to\mathbb Z[\mathfrak M]$ satisfying $ev_{\mathfrak M}([\mathcal A])=\mathcal A$ for each $\mathcal A\in\mathfrak M$.
\end{theorem}

\begin{proof}
We fix $\mathfrak M$ and drop the subscript $\mathfrak M$ from $ev_{\mathfrak M}$. We inductively define a compatible family of maps $\{ev^n:\V[n]\to\mathbb Z[\mathfrak M]\}_{n\geq0}$, where $ev^n$ factors through an injective group homomorphism $H_n\to\mathbb Z[\mathfrak M]$. By an abuse of notation, we also denote the group homomorphism by $ev^n$.

If $D\in\V[0]$ and $d(D)=d$, then the assignment $D\mapsto d\mathcal U$ clearly factors through an injective group homomorphism $H_0\cong\mathbb Z\to\mathbb Z[\mathfrak M]$, where $\mathcal U$ is the unique variety in $\mathfrak M_0$.

Assume by induction that $ev^{n-1}$ is a well-defined map on $\V[n-1]$ and that it factors through an injective group homomorphism $H_{n-1}\to\mathbb Z[\mathfrak M]$.

If $D\in\V[n-1]$, then define $ev^n(D):=ev^{n-1}(D)$ ensuring compatibility.

Let $A$ be an admissible variety of dimension $n$. Then there is an embedding $f:A\hookrightarrow\mathcal A$ for a unique $\mathcal A\in\mathfrak M_n$. Define $ev^n(A):=\mathcal A-ev^{n-1}(\mathcal A\setminus f(A))$.

To see that this definition does not depend on the choice of an embedding, let $g:A\hookrightarrow\mathcal A$ be another embedding. It will suffice to show that $ev^{n-1}(\mathcal A\setminus f(A))=ev^{n-1}(\mathcal A\setminus g(A))$. Note that the following equations hold in $H_n$.
\begin{eqnarray*}
\,[f(A)]&=&[g(A)],\\
\,[f(A)]+[\mathcal A\setminus f(A)]&=&[g(A)]+[\mathcal A\setminus g(A)].
\end{eqnarray*}
Hence $[\mathcal A\setminus f(A)]=[\mathcal A\setminus g(A)]$ in $H_n$. Under the hypothesis of a positive answer to Question \ref{LLQ}, we conclude the same equation in $H_{n-1}$.

If $\phi:A'\to A$ is a variety isomorphism, then $A'$ is admissible since $A$ is and both of them embed into the same variety $\mathcal A\in\mathfrak M_n$. Choosing an embedding $f$ of $A$ into $\mathcal A$ gives an embedding $f\circ\phi$ of $A'$ into $\mathcal A$. Since $f(A)=f\circ\phi(A')$, we have $ev^n(A)=ev^n(A')$.

If $A=A_1\sqcup A_2$ is a partition of an admissible variety $A$ of dimension $n$ into locally closed subvarieties and $\dim A_1=n$, then $A_1$ is admissible. Further if $f$ is an embedding of $A$ into $\mathcal A$, then using that $ev^{n-1}$ is additive we have
\begin{eqnarray*}
ev^n(A_1)+ev^n(A_2)&=&\mathcal A-ev^{n-1}(\mathcal A\setminus f(A_1))+ev^{n-1}(A_2)\\
&=&\mathcal A-ev^{n-1}((\mathcal A\setminus f(A))\sqcup f(A_2))+ev^{n-1}(A_2)\\
&=&\mathcal A-ev^{n-1}(\mathcal A\setminus f(A))-ev^{n-1}(f(A_2))+ev^{n-1}(A_2)\\
&=&\mathcal A-ev^{n-1}(\mathcal A\setminus f(A))\\
&=&ev^n(A).
\end{eqnarray*}

From the previous two paragraphs it follows that whenever $A$ is admissible and $A\doteqdot B$, then $ev^n(A)=ev^n(B)$.

Now let $D=\bigsqcup_{i\in[m]}D_i$ be an admissible partition of a variety $D$ of dimension $n$. Define $ev^n(D):=\sum_{i\in[m]}ev^n(D_i)$. Any two admissible partitions of $D$ admit a common admissible refinement and, as shown above, the value of $ev^n(D_i)$ does not change under refinements. Thus $ev^n(D)$ is independent of the choice of an admissible partition and is well-defined.

If $D\doteqdot D'$, then we choose partitions $D=\bigsqcup_{i\in[m]}D_i$ and $D'=\bigsqcup_{i\in[m]}D'_i$ such that $D_i$ is isomorphic to $D'_i$ for each $i$. By further refinements, we may as well assume that these partitions are admissible. Then $ev^n(D)=\sum_{i\in[m]}ev^n(D_i)=\sum_{i\in[m]}ev^n(D'_i)=ev^n(D')$.

This completes the proof that the map $ev^n$, uniquely determined by $\mathfrak M$, factors through an additive map $H_n\to\mathbb Z[\mathfrak M]$. It remains to show that $ev^n:H_n\to\mathbb Z[\mathfrak M]$ is injective.

We must show that if $D_1,D_2,D'_1,D'_2\in\V[n]$ satisfy $ev^n([D_1]-[D_2])=ev^n([D'_1]-[D'_2])$, then $[D_1]-[D_2]=[D'_1]-[D'_2]$. This claim can be restated as: $ev^n([D_1]+[D'_2])=ev^n([D_2]+[D'_1])$ implies $[D_1]+[D'_2]=[D_2]+[D'_1]$. Therefore it is sufficient to prove that if $D,D'\in\V[n]$ satisfy $ev^n([D])=ev^n([D'])$, then $[D]=[D']$ in $H_n$.

Let $D,D'\in\V[n]$ be such that $ev^n([D])=ev^n([D'])$. Looking at the ``$n$-dimensional component'' of this element of $\mathbb Z[\mathfrak M]$, we deduce that $\dim D=\dim D'$ and $d(D)=d(D')=:d$.

Suppose $D=\bigsqcup_{i\in[t]}D_i$ is an admissible partition of $D$. Without loss, we may assume that $\dim D_i=n$ if and only if $i\in[d]$. For each $i\in[d]$, let $f_i:D_i\to\mathcal A_i$ be a variety embedding, where $\mathcal A_i\in\mathfrak M_n$, and set $C_i:=\mathcal A_i\setminus f_i(D_i)$. Then $\dim C_i<n$ and $D_i\sqcup C_i\doteqdot\mathcal A_i$. Hence $ev^n([D_i])+ev^{n-1}([C_i])=\mathcal A_i$ for each $i\in[d]$.

Similarly starting with an admissible partition $D'=\bigsqcup_{i\in[s]}D'_i$, where $\dim D'_i=n$ if and only if $i\in[d]$, we obtain $\mathcal A'_i\in\mathfrak M_n$ and $C'_i$ such that $ev^n([D'_i])+ev^{n-1}([C'_i])=\mathcal A'_i$ for each $i\in[d]$. Now
\begin{eqnarray*}
\sum_{i\in[d]}\mathcal A_i+\sum_{i\in[t]\setminus[d]}ev^{n-1}([D_i])&+&\sum_{i\in[d]}ev^{n-1}([C'_i])\\
&=&ev^n([D])+\sum_{i\in[d]}ev^{n-1}([C_i])+ \sum_{i\in[d]}ev^{n-1}([C'_i])\\
&=& ev^n([D'])+\sum_{i\in[d]}ev^{n-1}([C'_i])+\sum_{i\in[d]}ev^{n-1}([C_i])\\
&=& \sum_{i\in[d]}\mathcal A'_i+\sum_{i\in[s]\setminus[d]}ev^{n-1}([D'_i])+\sum_{i\in[d]}ev^{n-1}([C_i]).
\end{eqnarray*}

By comparing the components of different dimensions, we get the following equations.
\begin{eqnarray*}
\sum_{i\in[d]}\mathcal A_i&=&\sum_{i\in[d]}\mathcal A'_i,\\
\sum_{i\in[t]\setminus[d]}ev^{n-1}([D_i])+\sum_{i\in[d]}ev^{n-1}([C'_i])&=&
\sum_{i\in[s]\setminus[d]}ev^{n-1}([D'_i])+\sum_{i\in[d]}ev^{n-1}([C_i]).
\end{eqnarray*}

The first equation gives that the list $\mathcal A_1,\mathcal A_2,\hdots,\mathcal A_d$ is the same as the list $\mathcal A'_1,\mathcal A'_2,\hdots,\mathcal A'_d$. Since the map $ev^{n-1}$ is injective, the second equation gives $\sum_{i\in[t]\setminus[d]}[D_i]+\sum_{i\in[d]}[C'_i]= \sum_{i\in[s]\setminus[d]}[D'_i]+\sum_{i\in[d]}[C_i]$ in $H_{n-1}$. Combining these, we obtain $[D]+\sum_{i\in[d]}[C_i]+\sum_{i\in[d]}[C'_i]= [D']+\sum_{i\in[d]}[C'_i]+\sum_{i\in[d]}[C_i]$ in $H_n$. Since $H_n$ is a group, cancelling common terms from both sides gives $[D]=[D']$. This completes the proof of injectivity of $ev^n$.

Define the map $ev:\V\to\mathbb Z[\mathcal M]$ by $ev(D):=ev^n(D)$ whenever $D\in\V[n]$. Compatibility of the family $\{ev^n\}$ gives that the map $ev$ is well-defined and factors through $\G$ to give an injective map $\G\to\mathbb Z[\mathfrak M]$.

Since, given $C,D\in\V$, there exists $n$ such that $C,D\in\V[n]$, the additivity of $ev^n:H_n\to\mathbb Z[\mathfrak M]$ for each $n$ implies that $ev:\G\to\mathbb Z[\mathfrak M]$ is a group homomorphism. The image of $ev$ generates the group $\mathbb Z[\mathfrak M]$ since the image of $\mathfrak M\subseteq\V$ generates the codomain. Hence $\G\cong\mathbb Z[\mathfrak M]$.
\end{proof}

In characteristic $0$, Hironaka's theorem on resolution of singularities allows us to choose a set $\mathfrak M$ of smooth representatives of birational equivalence classes in the theorem above. Therefore it is easy to see that the above theorem subsumes both presentations for $\G$ of Theorem \ref{Bittner} whenever $k$ is an algebraically closed field of characteristic $0$.

\section{The associated graded ring of $\K$}\label{monoidring}
We continue to work under the hypothesis of a positive answer to question \ref{LLQ} over an algebraically closed field of characteristic $0$ in this section. Under this hypothesis, the usual dimension function factorizes through the the Grothendieck group. 

Two varieties $X$ and $Y$ of dimension $n$ are birational if and only if there are open subvarieties $X'\subseteq X$ and $Y'\subseteq Y$ such that $X'\cong Y'$ if and only if $\dim([X]-[Y])=\dim([X\setminus X']-[Y\setminus Y'])<n$, where $[X]$ denotes the class of the variety $X$ in $\K$.

In general the product of two varieties in $\mathfrak M$ is birational (but not necessarily equal) to a variety in $\mathfrak M$. This suggests looking at the structure of the associated graded ring of $\K$, where the grading on $\K$ is induced by dimension. Let $\{F_n\}_{n\geq 0}$ be the filtration on $\K$ induced by dimensions and let $\mathfrak G$ denote the associated graded ring of $\K$ with respect to this filtration.

The construction of $\mathfrak G$ is as follows. Set $F_{-1}:=\{0\}$ for technical purposes. Let $G_n:=F_n/F_{n-1}$ for each $n\geq 0$ and let $\mathfrak G$ denote the abelian group $\bigoplus_{n\geq 0}G_n$. There are multiplication maps $G_n\times G_m\to G_{n+m}$ defined by $(x+F_{n-1})(y+F_{m-1})=xy+F_{n+m-1}$ for each $n,m\geq 0$. These maps combine to give a multiplication structure on $\mathfrak G$.

Let $\mathfrak B_n$ denote the set of birational equivalence classes of irreducible varieties of dimension $n$ and let $\mathfrak B:=\bigsqcup_{n\geq 0}\mathfrak B_n$. The set $\mathfrak B$ carries a monoid structure induced by the multiplication of varieties, where the class of a singleton acts as the identity. The usual dimension function on varieties factors through $\mathfrak B$.

\begin{theorem}
Suppose Question \ref{LLQ} admits a positive answer over an algebraically closed field $k$. The associated graded ring $\mathfrak G$ of $\K$ with respect to the dimension grading is the monoid ring $\mathbb Z[\mathfrak B]$, where $\mathfrak B$ is the multiplicative monoid of birational equivalence classes of irreducible varieties.
\end{theorem}

\begin{proof}
Since Question \ref{LLQ} admits a positive answer over $k$, we can use the group isomorphism of Theorem \ref{GGISFREE} induced by the evaluation map to define a multiplicative structure on $\mathbb Z[\mathfrak M]$. By an abuse of notation, we will say that $\{F_n\}_{n\geq 0}$ is a filtration on $\mathbb Z[\mathfrak M]$ and $\mathfrak G$ is its associated graded ring.

Let $\mathcal A\mapsto[[\mathcal A]]$ denote the canonical bijection $\mathfrak M\to\mathfrak B$, which takes an irreducible variety to its birational equivalence class. This clearly extends to a group isomorphism $\Phi:\mathfrak G\to\mathbb Z[\mathfrak B]$. We show that $\Phi$ also preserves multiplication.

Given $\mathcal A\in\mathfrak M_n$ and $\mathcal B\in\mathfrak M_m$, the product $\mathcal A\times_{\mathrm{Spec}\,k} \mathcal B$ is irreducible and thus is birational to a unique $\mathcal C\in\mathfrak M_{n+m}$. In other words, $(\mathcal A+F_{n-1})\cdotp(\mathcal B+F_{m-1})=\mathcal C+F_{n+m-1}$ in $\mathfrak G$. We also have $[[\mathcal C]]=[[\mathcal A\times_{\mathrm{Spec}\,k}\mathcal B]]=[[\mathcal A]][[\mathcal B]]$ in the monoid $\mathfrak B$. Hence $\Phi((\mathcal A+F_{n-1})\cdotp(\mathcal B+F_{m-1}))=\Phi(\mathcal A+F_{n-1})\cdotp\Phi(\mathcal B+F_{m-1})$. This shows that $\Phi$ preserves multiplication on the image of $\mathfrak M$ in $\mathfrak G$. It is routine to verify that $\Phi$ is multiplicative on the whole of $\mathfrak G$.
\end{proof}

\section{Further Remarks}
Recall that an element $a$ of a ring $R$ is said to be \textbf{regular} if it is not a zero divisor in $R$. The following question is important for better understanding of the Grothendieck ring and is open even in the case of algebraically closed fields.
\begin{que}
Let $k$ be a field of characteristic $0$. Is $\mathbb L:=[\mathbb A^1_k]$ a regular element of $\K$?
\end{que}
Lemma 4.8 and Proposition 4.9 in \cite{Seb} connect this question to Question \ref{LLQ} in special cases, but no further development has been made.

The model-theoretic Grothendieck ring $\mathrm{K}_0(k)$ of an algebraically closed field $k$, as defined in \cite{Kra}, is a quotient of $\K$. It is natural to ask the following question.
\begin{que}
Is the model-theoretic Grothendieck ring $\mathrm{K}_0(k)$ isomorphic to $\K$ for an algebraically closed field $k$?
\end{que}

\subsection*{Acknowledgements}
I would like to thank Prof.~Mike Prest for discussions and for careful reading of the paper. I would also like to thank Profs.~Thomas Scanlon, Fran\c{c}ois Loeser and Julien Sebag for useful comments on the first draft of the paper. I am grateful to Adam Biggs and Andrew Davies for discussions on Grothendieck rings and algebraic geometry.


\begin{thebibliography}{9999}
\bibitem{Bit} F. Bittner, The universal Euler characteristic for varieties of characteristic zero, Compos. Math., 140(4) (2004), 1011-1032.

\bibitem{Gro} M. Gromov, Endomorphisms of symbolic algebraic varieties, J. Eur. Math. Soc., 1(2) (1999), 109-197.

\bibitem{Kol} J. Koll\'{a}r, Conics in the Grothendieck ring, Adv. Math., 198(1) (2005), 27-35.

\bibitem{Kra} J. Krajic\v{e}k and T. Scanlon, Combinatorics with definable sets: Euler characteristics and Grothendieck rings, Bull. Symb. Logic, 6(3) (2000), 311-330.

\bibitem{LamSeb} S. Lamy and J. Sebag, Birational self-maps and piecewise algebraic geometry, J. Math. Sci. -Univ. Tokyo, 19(3) (2012), 325-357.

\bibitem{LL} M. Larsen and V. A. Lunts, Motivic measures and stable birational geometry, Mosc. Math. J., 3(1) (2003), 85-95.

\bibitem{LS} Q. Liu and J. Sebag, The Grothendieck ring of varieties and piecewise isomorphisms, Math. Z., 265(2) (2010), 321-342.

\bibitem{Poonen} B. Poonen, The Grothendieck ring of varieties is not a domain, Math. Res. Lett., 9(4) (2002), 493-497.

\bibitem{Sah} N. Sahasrabudhe, Grothendieck ring of varieties, Master thesis at Universit\'e de Bordeaux, available at http://www.algant.eu/documents/theses/neeraja.pdf.74, 2007.

\bibitem{Seb} J. Sebag, Variations on a question of Larsen and Lunts, Proc. Amer. Math. Soc., 138(4) (2010), 1231-1242.
\end{thebibliography}
\end{document}